\documentclass[amstex,12pt, amssymb]{article}

\usepackage{mathtext}
\usepackage[cp1251]{inputenc}
\usepackage[T2A]{fontenc}
\usepackage[dvips]{graphicx}
\usepackage{amsmath}
\usepackage{amssymb}
\usepackage{amsxtra}
\usepackage{latexsym}
\usepackage{ifthen}

\newcounter{lemma}[section]

\newcounter{corollary}[section]

\newcounter{remark}[section]

\newcounter{theorem}[section]

\newcounter{proposition}[section]

\newcounter{example}

\numberwithin{equation}{section}

\begin{document}

\markboth{R.\ R.\ Salimov, E.\ A.\ Sevost'yanov, V.\,A.
Targonskii}{\centerline{ ON MODULUS INEQUALITY ...}}

\def\cc{\setcounter{equation}{0}
\setcounter{figure}{0}\setcounter{table}{0}}

\overfullrule=0pt

%\normalsize\large

\author{R.\ R.\ Salimov, E.\ A.\ Sevost'yanov, V.\,A. Targonskii}

\title{
{\bf ON MODULUS INEQUALITY OF THE ORDER $P$ FOR THE INNER
DILATATION}}

\date{\today}
\maketitle

%\large
\begin{abstract}
The article is devoted to mappings with bounded and finite
distortion of plane domains. Our investigations are devoted to the
connection between mappings of the Sobolev class and upper bounds
for the distortion of the modulus of families of paths. For this
class, we have proved the Poletsky-type inequality with respect to
the so-called inner dilatation of the order~$p.$ Along the way, we
also obtained lower bounds for the modulus distortion under
mappings. We separately considered the situations of homeomorphisms
and mappings with branch points.
\end{abstract}


\section{Introduction}

This article is devoted to establishing estimates of the distortion
of the modulus of families of paths under mappings. The main object
of study is the Sobolev classes on the plane. Our manuscript refers
to the case when the inequality under study involves an inner
dilatation of an arbitrary order, in addition, the mapping can admit
branch points. Similar and close results of the authors may be found
in~\cite{LSS}--\cite{Sev$_5$}.

\medskip
Let us turn to the definitions and the formulation of the main
result. In what follows, $D$ is a domain in ${\Bbb C}$ and $dm(z)$
denotes the element of Lebesgue measure in ${\Bbb C}.$ As a rule, a
mapping $f:D\rightarrow{\Bbb C}$ is assumed to be {\it
sense-preserving,} moreover, we assume that $f$ has partial
derivatives almost everywhere. Put $z=x+iy,$ $i^2=-1,$
$f_{\overline{z}} = \left(f_x + if_y\right)/2$ and $f_z = \left(f_x
- if_y\right)/2.$ Note that the Jacobian of $f$ at $z\in D$ is
calculated by the formula
$$J(z,
f)=|f_z|^2-|f_{\overline{z}}|^2\,.$$
Given $p\geqslant 1,$ the {\it inner dilatation of the order $p$} is
defined as
\begin{equation}\label{eq1}
K_{I, p}(z,
f)=\frac{|f_z|^2-|f_{\overline{z}}|^2}{{(|f_z|-|f_{\overline{z}}|)}^{\,p}}
\end{equation}
for $J(z, f)\ne 0;$ an addition, we set $K_{I, p}(z, f)=1$ for
$f^{\,\prime}(z)=0$ and $K_{I, p}(z, f)=\infty$ otherwise.

\medskip
Recall the definition of Sobolev classes, which is the key to this
manuscript. In what follows, $C^k_0(U)$ denotes the space of
functions $u:U \rightarrow {\Bbb R} $ with a compact support in $U,$
having $k$ partial derivatives with respect to any variable that are
continuous in $U.$ We also recall the concept of a generalized
Sobolev derivative (see, for example,  \cite[Section~2, Ch.~I]{Re}).
Let $ U $ be an open set, $U\subset{\Bbb C},$ $u:U\rightarrow {\Bbb
R}$ is some function, $u \in L_{\rm loc}^{\,1}(U).$ Suppose there is
a function $v\in L_{\rm loc}^{\,1}(U)$ such that
$$\int\limits_U \frac{\partial \varphi}{\partial x_i}(z)u(z)\,dm(z)=
-\int\limits_U \varphi(z)v(z)\,dm(z)$$
for any function $\varphi\in C^1_{\,0}(U),$ $i=1,2.$ Then we say
that the function $v$ is a generalized derivative of the first order
of the function $u$ with respect to $x_i$ and denoted by the symbol:
$\frac{\partial u}{\partial x_i}(z):= v.$ Here $z=x_1+ix_2,$
$i^2=-1.$

A function $u\in W_{\rm loc}^{1,1}(U)$ if $u$ has generalized
derivatives of the first order with respect to each of the variables
in $U,$ which are locally integrable in $U.$

\medskip
A mapping $f:D\rightarrow {\Bbb C},$ $f(z)=u(z)+iv(z),$ belongs to
the Sobolev class $W_{\rm loc}^{1,1},$ write $f \in W^{1,1}_{\rm
loc}(D),$ if $u$ and $v$ have generalized partial derivatives of the
first order, which are locally integrable in $D$ in the first
degree. We write $f\in W^{1, k}_{\rm loc}(D),$ $k\in {\Bbb N},$ if
all coordinate functions $f=(f_1,\ldots, f_n) $ have generalized
partial derivatives of the first order, which locally integrable in
$D$ to the degree $k.$

\medskip
Recall that a mapping $f$ between domains $D$ and $D^{\,\prime}$ in
${\Bbb C}$ is of {\it finite distortion} if $f\in W^{1,1}_{\rm loc}$
and, besides that, there is a function $K(z)<\infty$ a.e. such that
$$%\label{eqOS1.3}
{\Vert f^{\,\prime}(z)\Vert}^2\leqslant K(z)\cdot J(z, f)
$$
for a.e. $z\in D,$ where $\Vert
f^{\,\prime}(z)\Vert=|f_z|+|f_{\overline{z}}|.$ For mappings of
finite distortion, we refer to \cite{IM} and to the reference
therein.

\medskip
Let $Q:{\Bbb C}\rightarrow {\Bbb R}$ be a Lebesgue measurable
function satisfying the condition $Q(z)\equiv 0$ for $z\in{\Bbb
C}\setminus D.$ Let $z_0\in\overline{D},$ $z_0\ne\infty,$
\begin{equation}\label{eq1ED}
B(z_0, r)=\{z\in {\Bbb C}: |z-z_0|<r\}\,,\qquad S(z_0,r) = \{
z\,\in\,{\Bbb C} : |z-z_0|=r\}\,,\end{equation}
\begin{equation}\label{eq1**A} A=A(z_0, r_1, r_2)=\{ z\,\in\,{\Bbb C} :
r_1<|z-z_0|<r_2\}\,.
\end{equation}
In what follows, $M_{\alpha}$ denotes the {\it $\alpha$-modulus} of
family $\Gamma$ of paths $\gamma: I\rightarrow {\Bbb C}$ in ${\Bbb
C},$ where $I$ is a closed, open or half-open interval in ${\Bbb R}$
(see \cite{Va}). Given $\alpha\geqslant 1,$ a mapping
$f:D\rightarrow \overline{\Bbb C}$ is called a {\it ring $Q$-mapping
at a point $z_0\in \overline{D}\setminus \{\infty\}$ with respect to
$\alpha$-modulus}, if the condition
\begin{equation} \label{eq2*!A}
M_{\alpha}(f(\Gamma(C_1, C_2, D)))\leqslant \int\limits_{A\cap D}
Q(z)\cdot \eta^{\,\alpha} (|z-z_0|)\, dm(z)
\end{equation}
holds for all $0<r_1<r_2<d_0:={\rm dist}\, (z_0,
\partial D),$ for any continua $C_1\subset
\overline{B(z_0, r_1)},$ $C_2\subset D\setminus B(z_0, r_2)$ and all
Lebesgue measurable functions $\eta:(r_1, r_2)\rightarrow [0,
\infty]$ such that
\begin{equation}\label{eq8BC}
\int\limits_{r_1}^{r_2}\eta(r)\,dr\geqslant 1\,.
\end{equation}
A mapping $f$ is called a {\it ring $Q$-mapping in $D$ with respect
to $\alpha$-modulus,} if condition~(\ref{eq2*!A}) is satisfied at
every point $z_0\in D,$ and a {\it ring $Q$-mapping in
$\overline{D}$ with respect to $\alpha$-modulus,} if the
condition~(\ref{eq2*!A}) holds at every point $z_0\in\overline{D}.$
For the properties of such mappings see~\cite{RSY} and~\cite{MRSY}.

\medskip
Recall that a pair $E=(A,\,C),$ where $A$ is an open set in ${\Bbb
R}^n,$ and $C$ is a compact subset of $A,$ is called {\it condenser}
in ${\Bbb R}^n$.  A quantity
\begin{equation}\label{eq4G}{\rm cap}_p\,E\quad=\quad{\rm
cap}_p\,(A,\,C)=\inf\limits_{u\,\in\,W_0\left(E\right)
}\quad\int\limits_A\,|\nabla u|^p\,\,dm(x)\,,
\end{equation}
where $dm(x)$ denotes the element of the Lebesgue measure in ${\Bbb
R}^n,$ $W_0(E)=W_0\left(A,\,C\right)$ is a family of all nonnegative
absolutely continuous on lines (ACL) functions $u:A\rightarrow {\Bbb
R}$ with compact support in $A$ and such that $u(x)\geqslant 1$ on
$C,$ is called {\it $p$-capacity} of the condenser $E$.

\medskip
The main results of the article are the following.

\begin{theorem}\label{th1}
Let $f:D\rightarrow {\Bbb C}$ be a homeomorphism with a finite
distortion and let $1<\alpha\leqslant 2.$ Assume that $K_{I,
\alpha}(z, f)\in L_{\rm loc}^1(D).$ Then $f$ satisfies the
relation~(\ref{eq2*!A}) at any point $z_0\in \overline{D}\setminus
\{\infty\}$ with $Q(z)=K_{I, \alpha}(z, f).$
\end{theorem}

\medskip
For a mapping $f\colon D\rightarrow{\Bbb R}^n$, a set $E\subset D$,
and $y\,\in\,{\Bbb R}^n$, we define the {\it multiplicity function}
$N(y, f, E)$ to be the number of preimages of $y$ in $E$, i.e.,
\begin{equation}\label{eq1.7A}
N(y, f, E)={{\rm card}}\left\{x\in E\,:\, f(x)=y\right\},\quad N(f,
E)=\sup_{y\in{\Bbb R}^n}\,N(y, f, E).
\end{equation}

\medskip
Let $X$ and $Y$ be metric spaces. A mapping $f:X\rightarrow Y$ is
{\it discrete} if $f^{\,-1}(y)$ is discrete for all $y\in Y$ and $f$
is {\it open} if $f$ maps open sets onto open sets. A mapping
$f:X\rightarrow Y$ is called {\it closed} if $f(A)$ is closed in
$f(X)$ whenever $A$ is closed in $X.$ For mappings with a branching,
we have the following.

\begin{theorem}\label{th1A}
Let $f:D\rightarrow {\Bbb C}$ be an open and discrete bounded
mapping such that $N(f, D)<\infty.$ Let $1< \alpha\leqslant 2$ and
let $z_0\in D.$ If $K_{I, \alpha}(z, f)\in L_{\rm loc}^1(D),$ then
$f$ satisfies the relation
$${\rm cap}_{\alpha}\, f(\mathcal{E})\leqslant\int\limits_{A} N^{\frac{1}{p-1}}(f, D)K_{I, \alpha}(z, f)\cdot
\eta^{\alpha}(|z-z_0|)\, dm(z)
$$
holds for $p=\frac{\alpha}{\alpha-1}$ and $\mathcal{E}=(B(z_0, r_2),
\overline{B(z_0, r_1)}),$ $A=A(z_0, r_1, r_2),$
$0<r_1<r_2<\varepsilon_0:={\rm dist}\,(z_0,
\partial D),$ and $\eta \colon (r_1,r_2)\rightarrow [0,\infty ]$ may
be chosen as arbitrary nonnegative Lebesgue measurable function
satisfying the relation~(\ref{eq8BC}).
\end{theorem}

\begin{theorem}\label{th2}
Let $f:D\rightarrow {\Bbb C}$ be an open, discrete and closed
mapping of a finite distortion and let $1<\alpha\leqslant 2.$ Assume
that $K_{I, \alpha}(z, f)\in L_{\rm loc}^1(D).$ Then, for any
$z_0\in \partial D,$ any $\varepsilon_0<d_0:=\sup\limits_{z\in
D}|z-z_0|$ and any compactum $C_2\subset D\setminus B(z_0,
\varepsilon_0)$ there is $\varepsilon_1,$
$0<\varepsilon_1<\varepsilon_0,$ such that the relation
\begin{equation}\label{eq3A}M_{\alpha}(f(\Gamma(C_1, C_2, D)))\leqslant \int\limits_{A(z_0,
\varepsilon, \varepsilon_1)}K_{I, \alpha}(z, f)
\eta^{\alpha}(|z-z_0|)\,dm(z)
\end{equation}
holds for any $\varepsilon\in (0, \varepsilon_1)$ and any
$C_1\subset \overline{B(z_0, \varepsilon)}\cap D,$ where $A(z_0,
\varepsilon, \varepsilon_1)$ is defined in~(\ref{eq1**A}), and
$\eta: (\varepsilon, \varepsilon_1)\rightarrow [0,\infty]$ is
arbitrary Lebesgue measurable function satisfying the
relation~(\ref{eq8BC}).
\end{theorem}

\section{On lower estimates of the modulus}

Let us give some important information concerning the relationship
between the moduli of the families of paths joining the sets and the
moduli of the families of the sets separating these sets. Mostly
this information can be found in Ziemer's publication,
see~\cite{Zi$_1$}. Let $G$ be a bounded domain in ${\Bbb R}^ n,$ and
$C_0, C_1$ are disjoint compact sets in $\overline{G}.$ Put $R=G
\setminus (C_{0} \cup C_{1})$ and $R^{\,*}=R \cup C_{0}\cup C_{1}.$
For a number $p>1 $ we define a {\it $p$ -capacity of the pair $C_0,
C_1 $ relative to the closure $G$} by the equality
$$C_p[G, C_0, C_1] = \inf \int\limits_{R} |\nabla u|^p\, dm(x),$$
where the exact lower bound is taken for all functions $u,$
continuous in $R^{\,*},$ $u\in ACL(R),$ such that $u=1$ on $C_1$ and
$u=0$ on $C_0.$ These functions are called {\it admissible} for
$C_p[G, C_0, C_1].$ We say that a set $\sigma \subset {\Bbb R}^n$
{\it separates} $C_0$ and $C_1$ in $R^{\,*},$ if $\sigma \cap R$ is
closed in $R$ and there are disjoint sets $A$ and $B,$ open relative
$R^{\,*}\setminus \sigma,$ such that $R^{\,*}\setminus \sigma=A\cup
B,$ $C_0\subset A$ and $C_1\subset B.$ Let $\Sigma$ denotes the
class of all sets that separate $C_0$ and $C_1$ in $R^{\,*}.$ For
the number $p^{\prime}=p/(p- 1)$ we define the quantity
\begin{equation}\label{eq13.4.12}
\widetilde{M_{p^{\prime}}}(\Sigma)=\inf\limits_{\rho\in
\widetilde{\rm adm} \Sigma} \int\limits_{{\Bbb
R}^n}\rho^{\,p^{\prime}}dm(x)
\end{equation}
where the notation $\rho\in \widetilde{\rm adm}\,\Sigma$ denotes
that $\rho$ is nonnegative Borel function in ${\Bbb R}^n$ such that
\begin{equation} \label{eq13.4.13}
\int\limits_{\sigma \cap R}\rho\, d{\mathcal H}^{n-1} \geqslant
1\quad\forall\, \sigma \in \Sigma\,. \end{equation}
Note that according to the result of Ziemer
\begin{equation}\label{eq3}
\widetilde{M_{p^{\,\prime}}}(\Sigma)=C_p[G , C_0 ,
C_1]^{\,-1/(p-1)}\,,
\end{equation}
see~\cite[Theorem~3.13]{Zi$_1$} for $p=n$ and \cite[p.~50]{Zi$_2$}
for $1<p<\infty,$ in addition, by the Hesse result
\begin{equation}\label{eq4}
M_p(\Gamma(E, F, D))= C_p[D, E, F]\,,
\end{equation}
where $(E \cup F)\cap
\partial D = \varnothing$ (see~\cite[Theorem~5.5]{Hes}). Shlyk has proved that the
requirement $(E \cup F)\cap
\partial D = \varnothing$ can be omitted, in other words, the equality~(\ref{eq4})
holds for any disjoint non-empty sets $E, F\subset \overline{D}$
(see~\cite[Theorem~1]{Shl}).

\medskip
Let $S$ be a surface, in other words, $S:D_s\rightarrow {\Bbb R}^n$
be a continuous mapping of an open set $D_s\subset {\Bbb R}^{n-1}.$
We put
$ N(y, S)={\rm card}\, S^{-1}(y)={\rm card} \{x\in D_s: S(x)=y\} $
and recall this function a {\it multiplicity function} of the
surface $S$ with respect to a point $y\in{\Bbb R}^n.$ Given a Borel
set $B\subset {\Bbb R}^n,$ its $(n-1)$-measured Hausdorff area
associated with the surface $S$ is determined by the formula
$ {\mathcal A}_S(B)={\mathcal A}_S^{n-1}(B)= \int\limits_B N(y, S)\,
d{\mathcal H}^{n-1} y, $
see~\cite[item~3.2.1]{Fe}. For a Borel function $\rho:\,{\Bbb
R}^n\rightarrow [0, \infty]$ its integral over the surface $S$ is
determined by the formula
$ \int\limits_S \rho\, d{\mathcal A}=\int\limits_{{\Bbb R}^n}
\rho(y) N(y, S)\,d{\mathcal H}^{n-1} y. $
In what follows, $J_kf(x)$ denotes the {\it $k$-dimensional
Jacobian} of the mapping $f$ at a point $x$ (see \cite[$\S\,3.2,$
Ch.~3]{Fe}).

\medskip
Let $n\geqslant 2,$ and let $\Gamma$  be a family of surfaces $S.$ A
Borel function $\rho\colon{\Bbb R}^n\rightarrow\overline{{\Bbb
R}^+}$ is called {\it an admissible} for $\Gamma,$ abbr.
$\rho\in{\rm adm}\,\Gamma,$ if
\begin{equation}\label{eq8.2.6}\int\limits_S\rho^{n-1}\,
d{\mathcal{A}}\geqslant 1\end{equation} for any $S\in\Gamma.$ Given
$p\in(0,\infty),$ a {\it $p$-modulus} of $\Gamma$ is called the
quantity
$$M_p(\Gamma)=\inf_{\rho\in{\rm adm}\,\Gamma} \int\limits_{{\Bbb
R}^n}\rho^p(x)\,dm(x)\,.$$ We also set
$M(\Gamma):=M_n(\Gamma).$
Let $p\geqslant 1.$ Let us say that some property $P$ holds for {\it
$p$-almost all surfaces} of the domain $D,$ if this property holds
for all surfaces in $D,$ except, maybe be, some of their subfamily,
$p$ -modulus of which is zero. If we are talking about the conformal
modulus $M(\Gamma):=M_n(\Gamma),$ the prefix ''$n$'' in the
expression ''$n$-almost all'', as a rule, is omitted. We say that a
Lebesgue measurable function $\rho\colon{\Bbb
R}^n\rightarrow\overline{{\Bbb R}^+}$ is {\it $p$-extensively
admissible} for the family $\Gamma$ of surfaces $S$ in ${\Bbb R}^n,$
abbr. $\rho\in{\rm ext}_p\,{\rm adm}\,\Gamma,$ if the
relation~(\ref{eq8.2.6}) is satisfied for $p$-almost all surfaces
$S$ of the family $\Gamma.$

\medskip
The next class of mappings is a generalization of quasiconformal
mappings in the sense of Gehring's ring definition (see
\cite{Ge$_3$}; it is the subject of a separate study, see, e.g.,
\cite[Chapter~9]{MRSY}). Let $D$ and $D^{\,\prime}$ be domains in
${\Bbb R}^n$ with $n\geqslant 2$. Suppose that $x_0\in\overline
{D}\setminus\{\infty\}$ and $Q\colon D\rightarrow(0,\infty)$ is a
Lebesgue measurable function. A function $f\colon D\rightarrow
D^{\,\prime}$ is called a {\it lower $Q$-mapping at a point $x_0$
relative to the $p$-modulus} if
\begin{equation}\label{eq1A}
M_p(f(\Sigma_{\varepsilon}))\geqslant \inf_{\rho\in{\rm ext}_p\,{\rm
adm}\Sigma_{\varepsilon}}\int\limits_{D\cap A(x_0, \varepsilon,
r_0)}\frac{\rho^p(x)}{Q(x)}\,dm(x)
\end{equation}
for every spherical ring $A(x_0, \varepsilon, r_0)=\{x\in {\Bbb
R}^n\,:\, \varepsilon<|x-x_0|<r_0\}$, $r_0\in(0,d_0)$,
$d_0=\sup_{x\in D}|x-x_0|$, where $\Sigma_{\varepsilon}$ is the
family of all intersections of the spheres $S(x_0, r)$ with the
domain $D$, $r\in (\varepsilon, r_0)$. If $p=n$, we say that $f$ is
a lower $Q$-mapping at $x_0$. We say that $f$ is a lower $Q$-mapping
relative to the $p$-modulus in $A\subset \overline {D}$ if
(\ref{eq1A}) is true for all $x_0\in A$.

\medskip The following statement can be proved much as Theorem 9.2 in \cite{MRSY}, so we omit the arguments.

\begin{lemma}\label{lem4} Let $D$,
$D^{\,\prime}\subset\overline{{\Bbb R}^n}$, let $x_0\in\overline
{D}\setminus\{\infty\}$, and let $Q$ be a Lebesgue measurable
function. A mapping $f\colon D\rightarrow D^{\,\prime}$ is a lower
$Q$-mapping relative to the $p$-modulus at a point $x_0$, $p>n-1$,
if and only if $M_p(f(\Sigma_{\varepsilon}))\geqslant
\int\limits_{\varepsilon}^{r_0} \frac{dr}{\|\,Q\|_{s}(r)}$ for all
$\varepsilon\in(0,r_0),\ r_0\in(0,d_0)$, $d_0=\sup_{x\in D}|x-x_0|$,
$s=\frac{n-1}{p-n+1}$, where, as above, $\Sigma_{\varepsilon}$
denotes the family of all intersections of the spheres $S(x_0, r)$
with $D$, $r\in (\varepsilon, r_0)$, $\|
Q\|_{s}(r)=(\int\limits_{D(x_0,r)}Q^{s}(x)\,d{\mathcal{A}})^{\frac{1}{s}}$
is the $L_{s}$-norm of $Q$ over the set $${D(x_0,r)=\{x\in D\,:\,
|x-x_0|=r\}=D\cap S(x_0,r)}\,.$$
\end{lemma}

\medskip
The following statement holds, cf.~\cite[Theorem~2.1]{KR$_1$},
\cite[Lemma~2.3]{SSP} and \cite[Lemma~2]{Sev$_5$}.

\begin{theorem}\label{thOS4.2}{\sl\, Let $p>1$ and let
$f:D\rightarrow {\Bbb C}$ be an open discrete mapping of a finite
distortion such that $N(f, D)<\infty.$ Then $f$ satisfies the
relation~(\ref{eq1A}) at any $z_0\in\overline{D}$ for $Q(z)=N(f,
D)\cdot K^{p-1}_{I, \alpha}(z, f),$ where $\alpha:=\frac{p}{p-1},$
$K_{I,\alpha}(z, f)$ is defined by (\ref{eq1}) and $N(f, D)$ is
defined in~(\ref{eq1.7A}).}
\end{theorem}

\begin{proof}
In many ways, the proof of this lemma uses the scheme outlined
in~\cite[Theorem~4]{Sev$_1$}. Observe that, $f=\varphi\circ g,$
where $g$ is some homeomorphism and $\varphi$ is an analytic
function, see (\cite[5.III.V]{St}). Thus, $f$ is differentiable
almost everywhere (see, e.g., \cite[Theorem~III.3.1]{LV}). Let $B$
be a Borel set of all points $z\in D,$ where $f$ has a total
differential $f^{\,\prime}(z)$ and $J(z, f)\ne 0$. Observe that, $B$
may be represented as a countable union of Borel sets $B_l$,
$l=1,2,\ldots\,,$ such that $f_l=f|_{B_l}$ are bi-lipschitzian
homeomorphisms (see \cite[items~3.2.2, 3.1.4 and 3.1.8]{Fe}).
Without loss of generality, we may assume that the sets $B_l$ are
pairwise disjoint. Denote by $B_*$ the set of all points $z\in D$ in
which $f$ has a total differential, however, $f^{\,\prime}(z)=0.$

Since $f$ is of finite distortion, $f^{\,\prime}(z)=0$ for almost
all $z,$ where $J(z, f)=0.$ Thus, by the construction the set
$B_0:=D\setminus \left(B\bigcup B_*\right)$ has a zero Lebesgue
measure. Thus, by~\cite[Theorem~9.1]{MRSY}, ${\mathcal
H}^{\,1}(B_0\cap S_r)=0$ for almost all circles $S_r:=S(z_0,r)$
centered at $z_0\in\overline{D},$ where, as usually, ${\mathcal
H}^{\,1}$ denotes the linear Hausdorff measure, and ''almost all''
means in the sense of $p$-modulus. Observe that, a function
$\psi(r):={\mathcal H}^{\,1}(B_0\cap S_r)$ is Lebesgue measurable by
the Fubini theorem, thus, the set $E=\{r\in {\Bbb R}: {\mathcal
H}^{\,1}(B_0\cap S_r)=0\}$ is Lebesgue measurable. Now, by Lemma~4.1
in \cite{IS} we obtain that ${\mathcal H}^{\,1}(B_0\cap S_r)=0$ for
almost any $r\in {\Bbb R}.$

\medskip
Denote by $D_*:=B(z_0, \varepsilon_0)\cap D\setminus\{z_0\},$
$0<\varepsilon_0<d_0=\sup\limits_{z\in D}|z-z_0|,$ and consider the
division of $D_*$ by the ring segments $A_k,$ $k=1,2,\ldots\,.$ Let
$\varphi_k$ be auxiliary quasiisometry which maps $A_k$ onto
rectangular $\widetilde{A_k}$ such that arcs of circles map onto
line segments. (For instance, we may take
$\varphi_k(\omega)=\log(\omega-z_0),$ $\omega\in A_k$). Consider a
family of mappings $g_k=f\circ\varphi_k^{\,-1},$
$g_k:\widetilde{A_k}\rightarrow {\Bbb C}.$ Observe that, $g_k\in
W^{1, 1}_{\rm loc}$ (see e.g. \cite[section~1.1.7]{Ma}). Thus,
$g_k\in ACL$ (see \cite[Theorems~1 and 2, item~I.1.1.3]{Ma}). Since
absolute continuity on a fixed segment implies $N$-property with a
respect to Lebesgue measure (see \cite[section~2.10.13]{Fe}), we
obtain that ${\mathcal H}^{\,1}((g_k\circ\varphi_k)(B_0\cap A_k\cap
S_r))={\mathcal H}^{\,1}(f(B_0\cap A_k\cap S_r))=0.$ Therefore, by
the subadditivity of the Hausdorff measure, we obtain that
${\mathcal H}^{\,1}(f(B_0\cap S_r))=0$ for almost all $r\in {\Bbb
R}.$

\medskip
Let us to show that ${\mathcal H}^{\,1}(f(B_*\cap S_r))=0$ for
almost any $r\in {\Bbb R}.$ Indeed, let $\varphi_k,$ $g_k$ and $A_k$
be such as above $A_k=\{z\in {\Bbb C}: z-z_0=re^{i\varphi}, r\in
(r_{k-1}, r_{k}), \varphi\in (\psi_{k-1}, \psi_k)\},$ and let
$S_k(r)$ be a part of the sphere $S(z_0, r)$ which belongs to the
segment $A_k,$ i.e. $S_k(r)=\{z\in {\Bbb C}: z-z_0=re^{i\varphi},
\varphi\in (\psi_{k-1}, \psi_k)\}.$ By the construction, $\varphi_k$
maps $S_k(r)$ onto the segment $I(k, r)=\{z\in {\Bbb C}: z=\log
r+it, t\in (\psi_{k-1}, \psi_k).$ Applying \cite[theorem~3.2.5]{Fe},
we obtain that
$${\mathcal H}^{\,1}(g_k(\varphi_k(B_*\cap S_k(r))))\leqslant
\int\limits_{g_k(\varphi_k(B_*\cap S_k(r)))}N(y, g_k,
\varphi_k(B_*\cap S_k(r)))d{\mathcal H}^{1}y=$$
$$=\int\limits_{\varphi_k(B_*\cap
S_k(r))}|g_k^{\,\prime}(r+te)|dt=0$$
for almost all $r\in (r_{k-1}, r_{k}).$ Thus, ${\mathcal
H}^{\,1}(f(B_*\cap S_k(r)))=0$ for almost all $r\in (r_{k-1},
r_{k}).$ By subadditivity of the Hausdorff measure we obtain that
${\mathcal H}^{\,1}(f(B_*\cap S_r))=0$ for almost all $r\in {\Bbb
R},$ as required.

\medskip
Let $\Gamma$ be a family of all intersections of circles $S_r$,
$r\in(\varepsilon,\varepsilon_0)$,
$\varepsilon_0<d_0=\sup\limits_{z\in D}\,|z-z_0|,$ with a domain
$D.$ Given an admissible function $\rho_*\in{\rm adm}\,f(\Gamma),$
$\rho_*\equiv 0$ outside of $f(D)$, we set $\rho\equiv 0$ outside of
$D$ and on $B_0,$ and
$$\rho(z)\colon=\rho_*(f(z))\left(\frac{|f_z|^2-|f_{\overline{z}}|^2}
{|f_z|-|f_{\overline{z}}|}\right) \qquad\text{for}\ z\in D\setminus
B_0\,.$$

Given a fixed domain $D_{r}^{\,*}\in f(\Gamma),$
$D_{r}^{\,*}=f(S_r\cap D),$ observe that
$$D_{r}^{\,*}=\bigcup\limits_{i=0}^{\infty} f(S_r\cap
B_i)\bigcup f(S_r\cap B_*)\,,$$
and, consequently, for almost all $r\in (0, \varepsilon_0),$ we
obtain that
\begin{equation}\label{eq10a}
1\leqslant \int\limits_{D^{\,*}_r}\rho_*(y)d{\mathcal A_*}=
\sum\limits_{i=0}^{\infty} \int \limits_{f(S_r\cap B_i)} N (y,
S_r\cap B_i)\rho_*(y) d{\mathcal H}^{1}y +
\end{equation}
$$+\int\limits_{f(S_r\cap B_*)}N (y,
S_r\cap B_*) \rho_*(y) d{\mathcal H}^{\,1}y\,.$$ Due to the above,
by~(\ref{eq10a}) we obtain that
\begin{equation}\label{eq11A}
1\leqslant \int\limits_{D^{\,*}_r}\rho_*(y)d{\mathcal A_*}=
\sum\limits_{i=1}^{\infty} \int \limits_{f(S_r\cap B_i)}N (y,
S_r\cap B_i) \rho_*(y) d{\mathcal H}^{\,1}y
\end{equation}
for almost all $r\in (0, \varepsilon_0).$
Observe that,
$l(f^{\,\prime}(z)):=\min\limits_{|h|=1}|f^{\,\prime}(z)h|=|f_z|-|f_{\overline{z}}|,$
see \cite[relation~(10).A.I]{A}. Now, arguing on each~$B_i$,
$i=1,2,\ldots$, by \cite[item~1.7.6 and Theorem~3.2.5]{Fe} we obtain
that
$$\int\limits_{B_i\cap S_r}\rho\,d{\mathcal A}=
\int\limits_{B_i\cap
S_r}\rho_*(f(z))\left(\frac{|f_z|^2-|f_{\overline{z}}|^2}
{|f_z|-|f_{\overline{z}}|}\right)\,d{\mathcal A}=$$
$$=\int\limits_{B_i\cap S_r}\rho_*(f(z))\cdot
\left(\frac{|f_z|^2-|f_{\overline{z}}|^2}
{|f_z|-|f_{\overline{z}}|}\right)\cdot \frac{1}{\frac{d{\mathcal
A_*}}{d{\mathcal A}}}\cdot \frac{d{\mathcal A_*}}{d{\mathcal
A}}\,d{\mathcal A}\geqslant \int\limits_{B_i\cap
S_r}\rho_*(f(z))\cdot \frac{d{\mathcal A_*}}{d{\mathcal
A}}\,d{\mathcal A}=$$
\begin{equation}\label{eq12A}
=\int\limits_{f(B_i\cap S_r)}\rho_{*}\,d{\mathcal A_*}
\end{equation} for almost any $r\in (0, \varepsilon_0).$
By~(\ref{eq11A}) and (\ref{eq12A}), and by~\cite[Lemma~4.1]{IS} we
obtain that $\rho\in{\rm{ext}}_p{\rm\,adm}\,\Gamma.$

\medskip
Using the change of variables on $B_l$, $l=1,2,\ldots$ (see, e.g.,
\cite[Theorem~3.2.5]{Fe}), by the countable additivity of the
Lebesgue integral we obtain that
$$\int\limits_{D}\frac{\rho^p(z)}{K^{p-1}_{I,
\frac{p}{p-1}}(z)}\,dm(z)=\sum\limits_{l=1}^{\infty}\int\limits_{B_l}
\frac{\rho^p(z){(|f_{z}|-|f_{\overline{z}}|)}^{p}}{{|J(z,
f)|}^{p-1}}\,dm(z)=$$
$$=\sum\limits_{l=1}^{\infty}\int\limits_{B_l}
\frac{{(|f_{z}|-|f_{\overline{z}}|)}^{p}}{{|J(z, f)|}^{p-1}}\cdot
\rho^p_*(f(z))\left(\frac{{(|f_z|^2-|f_{\overline{z}}|^2)}^{p}}
{{(|f_z|-|f_{\overline{z}}|)}^{p}}\right)\,\,dm(z)=$$
$$=\sum\limits_{l=1}^{\infty}\int\limits_{B_l}\rho^p_*(f(z))|J(z, f)|\,dm(z)
=\sum\limits_{l=1}^{\infty}\int\limits_{f(B_l)}\rho^p_*(y)\,dm(y)
\leqslant \int\limits_{f(D)}N(f, D)\rho^p_*(y)\,dm(y)\,,$$ as
required.
\end{proof}

We also need the following statement given in
\cite[Proposition~10.2, Ch.~II]{Ri}.

\begin{proposition}\label{pr1}
Let $E=(A,\,C)$ be a condenser in ${\Bbb R}^n$ and let $\Gamma_E$ be
the family of all paths of the form $\gamma:[a,\,b)\rightarrow A$
with $\gamma(a)\in C$ and $|\gamma|\cap\left(A\setminus
F\right)\ne\varnothing$ for every compact $F\subset A.$ Then ${\rm
cap}_q\,E= M_q\left(\Gamma_E\right).$
\end{proposition}

\medskip
An analogue of the following assertion has been proved several times
earlier under slightly different conditions,
see~\cite[Lemma~4.2]{SevSkv$_1$} and \cite[Lemma~5]{Sev$_1$}. In the
formulation given below, this result is proved for the first time.

 \begin{lemma}\label{l4.4}
{ Let $D$ be a domain in ${\Bbb R}^n,$ $n\geqslant 2,$ let $p>n-1,$
let $x_0\in D$ and let $f:D\rightarrow {\Bbb R}^n$ be an open and
discrete mapping satisfying the relation~(\ref{eq1A}) at a point
$x_0.$ Assume that $Q\colon D\rightarrow[0,\infty]$ is a Lebesgue
measurable function which is locally integrable in the degree
$s=\frac{n-1}{p-n+1}$ in $D.$ Then the relation
\begin{equation}\label{eq4A}
{\rm cap}_{\alpha}\, f(\mathcal{E})\leqslant\int\limits_{A}
Q^{\,*}(x)\cdot \eta^{\alpha}(|x-x_0|)\, dm(x)
\end{equation}
holds for $\alpha=\frac{p}{p-n+1}$ and
$Q^{\,*}(x)=Q^{\frac{n-1}{p-n+1}}(x),$ where $\mathcal{E}=(B(x_0,
r_2), \overline{B(x_0, r_1)}),$ $A=A(x_0, r_1, r_2),$
$0<r_1<r_2<\varepsilon_0:={\rm dist}\,(x_0,
\partial D),$ and $\eta \colon (r_1,r_2)\rightarrow [0,\infty ]$ may
be chosen as arbitrary nonnegative Lebesgue measurable function
satisfying the relation~(\ref{eq8BC}).}
 \end{lemma}

\begin{proof} Observe that $s=\alpha-1.$ By Lemma~2 in~\cite{SalSev}, it is sufficiently to prove that
$${\rm cap_{\alpha}}\,
f(\mathcal{E})\leqslant \frac{\omega_{n-1}}{I^{*\,\alpha-1}}\,,$$
where $\mathcal{E}$ is a condenser $\mathcal{E}=(B(x_0, r_2),
\overline{B(x_0, r_1)}),$ and $q^{\,*}_{x_0}(r)$ denotes the
integral average of $Q^{\alpha-1}(x)$ under $S(x_0, r),$
\begin{equation}\label{eq16}
q_{x_0}(r)=\frac{1}{\omega_{n-1}r^{n-1}}\int\limits_{S(x_0,
r)}Q(x)\,d\mathcal{H}^{n-1}\,,
\end{equation}
$\omega_{n-1}$ denotes the area of the unit sphere ${\Bbb S}^{n-1}$
in ${\Bbb R}^n$ and $$I^{\,*}=I^{\,*}(x_0,
r_1,r_2)=\int\limits_{r_1}^{r_2}\
\frac{dr}{r^{\frac{n-1}{\alpha-1}}q^{\,*\,\frac{1}{\alpha-1}}_{x_0}(r)}\,.$$
Let $\varepsilon\in (r_1, r_2)$ and let $B(x_0, \varepsilon).$ We
put $C_0=\partial f(B(x_0, r_2)),$ $C_1=f(\overline{B(x_0, r_1)}),$
$\sigma=\partial f(B(x_0, \varepsilon)).$ Since $f$ is continuous in
$D,$ the set $f(B(x_0, r_2))$ is bounded.

\medskip
Since $f$ is continuous, $\overline{f(B(x_0, r_1))}$ is a compact
subset of $f(B(x_0, \varepsilon)),$ and $\overline{f(B(x_0,
\varepsilon))}$ is a compact subset of $f(B(x_0, r_2)).$ In
particular,
$$\overline{f(B(x_0, r_1))}\cap
\partial f(B(x_0, \varepsilon))=\varnothing\,.$$
Let, as above, $R=G \setminus (C_{0} \cup C_{1}),$ $G:=f(D),$ and
$R^{\,*} = R \cup C_{0}\cup C_{1}.$ Then $R^{\,*}.$ Observe that,
$\sigma$ separates $C_0$ from $C_1$ in $R^{\,*}=G.$ Indeed, the set
$\sigma \cap R$ is closed in $R,$ besides that, if $A:=G\setminus
\overline{f(B(x_0, \varepsilon))}$ and $B= f(B(x_0, \varepsilon)),$
then $A$ and $B$ are open in $G\setminus \sigma,$ $C_0\subset A,$
$C_1\subset B$ and $G\setminus \sigma=A\cup B.$

\medskip
Let $\Sigma$ be a family of all sets, which separate $C_0$ from
$C_1$ in $G.$ Below by $\bigcup\limits_{r_1<r<r_2}
\partial f(B(x_0, r))$ or $\bigcup\limits_{r_1<r<r_2}
f(S(x_0, r))$ we mean the union of all Borel sets into a family, but
not in a theoretical-set sense (see \cite[item~3, p.~464]{Zi$_1$}).
Let $\rho^{n-1}\in \widetilde{{\rm adm}}\bigcup\limits_{r_1<r<r_2}
\partial f(B(x_0, r))$ in the sense of the relation~(\ref{eq13.4.13}). Then $\rho\in {\rm
adm}\bigcup\limits_{r_1<r<r_2}
\partial f(B(x_0, r))$ in the sense of~(\ref{eq8.2.6}).
By the openness of the mapping $f$ we obtain that  $\partial
f(B(x_0, r))\subset f(S(x_0, r)),$  therefore, $\rho\in {\rm
adm}\bigcup\limits_{r_1<r<r_2} f(S(x_0, r))$ and, consequently,
by~(\ref{eq13.4.12})
\begin{multline}\label{eq5E}
\widetilde{M_{\frac{p}{n-1}}}(\Sigma)\geqslant
\widetilde{M_{\frac{p}{n-1}}}\left(\bigcup\limits_{r_1<r<r_2}
\partial f(B(x_0, r))\right)\geqslant\\
\geqslant
\widetilde{M_{\frac{p}{n-1}}}\left(\bigcup\limits_{r_1<r<r_2}
f(S(x_0, r))\right)\geqslant\\
\geqslant M_p\left(\bigcup\limits_{r_1<r<r_2} f(S(x_0, r))\right).
 \end{multline}
However, by~(\ref{eq3}) and ~(\ref{eq4}) we obtain that
\begin{equation}\label{eq6D}
\frac{1}{(M_{\alpha}(\Gamma(C_0, C_1, G)))^{1/(\alpha-1)}}=
\widetilde{M_{\frac{p}{n-1}}}(\Sigma)\,.
\end{equation}
Let $\Gamma_{f(\mathcal{E})}$ be a family of all paths which
correspond to the condenser $\varphi(f(\mathcal{E}))$ in the sense
of Proposition~\ref{pr1}, and let $\Gamma^{\,*}_{f(\mathcal{E})}$ be
a family of all rectifiable paths of $\Gamma_{f(\mathcal{E})}.$ Now,
observe that, the families $\Gamma^{*}_{f(\mathcal{E})}$ and $\Gamma
(C_0, C_1, G)$ have the same families of admissible functions
$\rho.$ Thus,
$$M_{\alpha}(\Gamma_{f(\mathcal{E})})=M_{\alpha}(\Gamma(C_0, C_1, G))\,.$$
By Proposition~\ref{pr1}, we obtain that
$M_{\alpha}(\Gamma_{f(\mathcal{E})})={\rm
cap}_{\alpha}f(\mathcal{E}).$ By~(\ref{eq6D}) we obtain that
\begin{equation}\label{eq7B}
\left(\widetilde{M_{\frac{p}{n-1}}}(\Sigma)\right)^{\alpha-1}=\frac{1}{{\rm
cap}_{\alpha}f(\mathcal{E})}\,.
\end{equation}
Finally, by~(\ref{eq5E}) and~(\ref{eq7B}) we obtain that
$$
{\rm cap}_{\alpha}f(\mathcal{E}) \leqslant \frac{1}{M_{\alpha}
\left(\bigcup\limits_{r_1<r<r_2} f(S(x_0, r))\right)^{\alpha-1}}\,.
$$
By Lemma~\ref{lem4}, we obtain that
$$
{\rm cap}_{\alpha}f(\mathcal{\mathcal{E}}) \leqslant
\frac{1}{\left(\int\limits_{r_1}^{r_2} \frac{dr}{\Vert
\,Q\Vert_{s}(r)}\right)^{s}}=\frac{1}{I^{*\,\alpha-1}}\,,
$$
as required.
\end{proof}

\section{Proof of main results}

The following result is proved in~\cite[Theorem~5]{Sev$_2$}.

\begin{proposition}\label{pr2}
{\, Let $x_0\in \partial D,$ let $f:D\rightarrow {\Bbb R}^n$ be an
open, discrete and closed bounded lower $Q$-mapping with a respect
to $p$-modulus in $D\subset{\Bbb R}^n,$ $Q\in
L_{loc}^{\frac{n-1}{p-n+1}}({\Bbb R}^n),$ $n-1<p,$ and
$\alpha:=\frac{p}{p-n+1}.$ Then for any
$\varepsilon_0<d_0:=\sup\limits_{x\in D}|x-x_0|$ and any compactum
$C_2\subset D\setminus B(x_0, \varepsilon_0)$ there is
$\varepsilon_1,$ $0<\varepsilon_1<\varepsilon_0,$ such that the
inequality
\begin{equation}\label{eq3B}
M_{\alpha}(f(\Gamma(C_1, C_2, D)))\leqslant \int\limits_{A(x_0,
\varepsilon, \varepsilon_1)}Q^{\frac{n-1}{p-n+1}}(x)
\eta^{\alpha}(|x-x_0|)\,dm(x)\,,
\end{equation}
holds for any $\varepsilon\in (0, \varepsilon_1)$ and any compactum
$C_1\subset \overline{B(x_0, \varepsilon)}\cap D,$ where $A(x_0,
\varepsilon, \varepsilon_1)=\{x\in {\Bbb R}^n:
\varepsilon<|x-x_0|<\varepsilon_1\}$ and $\eta: (\varepsilon,
\varepsilon_1)\rightarrow [0,\infty]$ is any nonnegative Lebesgue
measurable function such that
\begin{equation}\label{eq6B}
\int\limits_{\varepsilon}^{\varepsilon_1}\eta(r)\,dr=1\,.
\end{equation}
 }
\end{proposition}

\begin{remark}
Note that, if~(\ref{eq2*!A}) holds for any function $\eta$ with a
condition (\ref{eq6B}), then the same relationship holds for any
function $\eta$ with the condition~(\ref{eq8BC}). Indeed, let $\eta$
be a nonnegative Lebesgue function that satisfies the condition
(\ref{eq8BC}). If $J:=\int\limits_{r_1}^{r_2}\eta(t)\,dt<\infty,$
then we put $\eta_0:=\eta/J.$ Obviously, the function $\eta_0$
satisfies condition~(\ref{eq6B}). Then the relation~(\ref{eq3B})
gives that
$$M_{\alpha}(f(\Gamma(C_1, C_2, D)))\leqslant
$$
\begin{equation}\label{eq1X}
\frac{1}{J^{\alpha}}\int\limits_A Q(x)\cdot
\eta^{\alpha}(|x-x_0|)\,dm(x)\leqslant \int\limits_A Q(x)\cdot
\eta^{\alpha}(|x-x_0|)\,dm(x)
\end{equation}
because $J\geqslant 1.$ Let now $J=\infty.$ Then, by
\cite[Theorem~I.7.4]{Sa}, a function $\eta$ is a limit of a
nondecreasing nonnegative sequence of simple functions $\eta_m,$
$m=1,2,\ldots .$ Set
$J_m:=\int\limits_{r_1}^{r_2}\eta_m(t)\,dt<\infty$ and
$w_m(t):=\eta_m(t)/J_m.$ Then, similarly to~(\ref{eq1X}) we obtain
that
$$M_{\alpha}(f(\Gamma(C_1, C_2, D)))\leqslant
$$
\begin{equation}\label{eq11B}
\frac{1}{J_m^{\alpha}}\int\limits_A Q(x)\cdot
\eta_m^{\alpha}(|x-x_0|)\,dm(x)\leqslant \int\limits_A Q(x)\cdot
\eta_m^{\alpha}(|x-x_0|)\,dm(x)\,,
\end{equation}
because $J_m\rightarrow J=\infty$ as $m\rightarrow\infty$
(see~\cite[Lemma~I.11.6]{Sa}). Thus, $J_m\geqslant 1$ for
sufficiently large $m\in {\Bbb N}.$ Observe that, a functional
sequence $\varphi_m(x)=Q_*(x)\cdot \eta_m^{\alpha}(|x-x_0|),$
$m=1,2\ldots ,$ is nonnegative, monotone increasing and converges to
a function $\varphi(x):=Q_*(x)\cdot \eta^{\alpha}(|x-x_0|)$ almost
everywhere. By the Lebesgue theorem on the monotone convergence
(see~\cite[Theorem~I.12.6]{Sa}), it is possible to go to the limit
on the right side of the inequality~(\ref{eq11B}), which gives us
the desired inequality~(\ref{eq2*!A}).
\end{remark}

\medskip
The following result is proved in~\cite[Theorem~6]{Sev$_5$}.

\begin{proposition}\label{pr3}
{\, Let $x_0\in \partial D,$ and let $f:D\rightarrow {\Bbb R}^n$ be
a bounded lower $Q$-homeomorphism with respect to $p$-modulus in a
domain $D\subset{\Bbb R}^n,$ $Q\in L_{\rm
loc}^{\frac{n-1}{p-n+1}}({\Bbb R}^n),$ $p>n-1$ and
$\alpha:=\frac{p}{p-n+1}.$ Then $f$ is a ring
$Q^{\frac{n-1}{p-n+1}}$-homeomorphism  with respect to
$\alpha$-modulus at this point, where $\alpha:=\frac{p}{p-n+1}.$
 }
\end{proposition}

\medskip
{\it Proof of Theorem~\ref{th1}.} Fix $z_0\in \overline{D}.$ Two
situations are possible: when $z_0\in D,$ and when $z_0\in \partial
D.$ Let $z_0\in D.$ Set $p:=\frac{\alpha}{\alpha-1}.$ Due to
Theorem~\ref{thOS4.2}, $f$ satisfies the relation~(\ref{eq1A}) with
$Q(z):=K_{I, \alpha}^{p-1}(z, f).$ Now, $f$ satisfies the
relation~(\ref{eq4A}) with $Q^{\,*}(z):=K_{I, \alpha}(z, f).$
Observe that the relation
\begin{equation}\label{eq5}
M_{\alpha}(f(\Gamma(C_1, C_2, D)))\leqslant {\rm
cap}_{\alpha}(f(B(z_0, r_2)), f(\overline{B(z_0, r_1)}))
\end{equation}
holds for all $0<r_1<r_2<d_0:={\rm dist}\, (z_0,
\partial D)$ and for any
continua $C_1\subset \overline{B(z_0, r_1)},$ $C_2\subset D\setminus
B(z_0, r_2).$ Indeed, $f(\Gamma(C_1, C_2, D))>\Gamma_{f(E)},$ where
$E:=(f(B(z_0, r_2)), f(\overline{B(z_0, r_1)})),$ and
$\Gamma_{f(E)}$ is a family from Proposition~\ref{pr1} for the
condenser $f(E).$ The relation~(\ref{eq5}) finishes the proof for
the case $z_0\in D.$ Let now $z_0\in \partial{D}.$ Again, by
Theorem~\ref{thOS4.2}, $f$ satisfies the relation~(\ref{eq1A}) with
$Q(z):=K_{I, \alpha}^{p-1}(z, f).$ Now $f$ satisfies the
relation~(\ref{eq2*!A}) with $Q^{\,*}(z):=K_{I, \alpha}(z, f)$ by
Proposition~\ref{pr3}.~$\Box$

\medskip
{\it Proof of Theorem~\ref{th1A}} directly follows by
Theorem~\ref{thOS4.2} and Lemma~\ref{l4.4}.~$\Box$

\medskip
{\it Proof of Theorem~\ref{th2}} directly follows by
Theorem~\ref{thOS4.2} and Proposition~\ref{pr3}.~$\Box$

\section{Example}

Let $Q(z)=\log\frac{e}{|z|},$ $z\in {\Bbb D},$ let
$q_0(r):=\log\frac{e}{r},$ and let $1<\alpha< 2.$ Observe that
$\int\limits_{\varepsilon}^{\varepsilon_0}\frac{dt}{t^{\frac{1}{\alpha-1}}q_{0}^{\,\frac{1}{\alpha-1}}(t)}<\infty$
for any $\varepsilon_0\in (0, 1)$ and any $\varepsilon\in (0,
\varepsilon_0),$ in addition,
\begin{equation}\label{eq2B}
\int\limits_0^{\varepsilon_0}\frac{dt}{t^{\frac{1}{\alpha-1}}q_{0}^{\,\frac{1}{\alpha-1}}(t)}=\infty\,,
\end{equation}
where $q_0(t)$ is defined by~(\ref{eq16}). Set
$$
f(z)=\frac{z}{|z|}\rho(|z|)\,,\quad z\in {\Bbb D}\setminus \{0\},
\quad f(0):=0\,,
 $$
where
$$
\rho(|z|)=\left(1+\frac{2-\alpha}{\alpha-1}\int\limits_{|z|}^1
\frac{dt}{t^{\frac{1}{\alpha-1}}q_0^{\frac{1}{\alpha-1}}(t)}\right)^{\frac{\alpha-1}{\alpha-2}}\,.
 $$
Observe that, $f\in ACL$ and, besides that, $f$ is differentiable in
${\Bbb D}$ almost everywhere.
By the technique used in \cite[Proposition~6.3]{MRSY}
$$|J(z, f)|=\delta_{\tau}(z)\cdot \delta_r(z), \qquad \Vert f^{\,\prime}(z)\Vert= \max\{\delta_{\tau}(z), \delta_r(z)
\}\,,\quad l(f^{\,\prime}(z))= \min\{\delta_{\tau}(z), \delta_r(z)
\}\,,$$
where $l(f^{\,\prime}(z))=|f_z|-|f_{\overline{z}}|,$ $\Vert
f^{\,\prime}(z)\Vert=|f_z|+|f_{\overline{z}}|,$
$$\delta_{\tau}(z)=\frac{|f(z)|}{|z|}\,,\quad \delta_r(z)=\frac{|\partial |f(z)||}{\partial |z|}\,.$$
Thus,
 \begin{gather*}
\Vert f^{\,\prime}(z)\Vert
=\left(1+\frac{2-\alpha}{\alpha-1}\int\limits_{|z|}^1
\frac{dt}{t^{\frac{1}{\alpha-1}}q_0^{\frac{1}{\alpha-1}}(t)}\right)^{\frac{\alpha-1}{\alpha-2}}\frac{1}
{|z|}\,,\\
l(f^{\,\prime}(z))=\left(1+\frac{2-\alpha}{\alpha-1}\int\limits_{|z|}^1
\frac{dt}{t^{\frac{1}{\alpha-1}}q_0^{\frac{1}{\alpha-1}}(t)}\right)^{\frac{1}{\alpha-2}}\frac{1}
{|z|^{\frac{1}{\alpha-1}}q_0^{\frac{1}{\alpha-1}}(|z|)}
 \end{gather*}
and
 $$
|J(z, f)|=\left(1+\frac{2-\alpha}{\alpha-1}\int\limits_{|z|}^1
\frac{dt}{t^{\frac{1}{\alpha-1}}q_0^{\frac{1}{\alpha-1}}(t)}\right)^{\frac{\alpha}{\alpha-2}}\frac{1}
{|z|^{1+\frac{1}{\alpha-1}}q_0^{\frac{1}{\alpha-1}}(|z|)}\,.
 $$
Observe that, $f\in W_{\rm loc}^{1, \alpha}$ for $\alpha>1.$ Indeed,
$\Vert f^{\,\prime}(z)\Vert$ is bounded outside of some neighborhood
of the origin, in addition, for sufficiently small $r>0$ and some
$C>0,$ we obtain that $\Vert f^{\,\prime}(z)\Vert\leqslant C/|z|.$
Besides that, by the Fubini theorem $\int\limits_{B(0, r)}\Vert
f^{\,\prime}(z)\Vert^{\,\alpha}\,dm(z)\leqslant
C^{\,\alpha}2\pi\cdot\int\limits_0^r r^{1-\alpha}dr<\infty.$ Let $K$
be any compact set in ${\Bbb D}.$ Then there is $R>0$ such that
$K\subset B(0, R).$ Due to the above, we obtain that
$$\int\limits_K\Vert f^{\,\prime}(z)\Vert^{\,\alpha}\,dm(z)\leqslant
\int\limits_{B(0, r)}\Vert
f^{\,\prime}(z)\Vert^{\,\alpha}\,dm(z)+$$$$+\int\limits_{B(0,
R)\setminus B(0, r)}\Vert
f^{\,\prime}(z)\Vert^{\,\alpha}\,dm(z)<\infty\,,$$
therefore, $f\in W_{\rm loc}^{1, \alpha}.$
Observe that $K_{I, \alpha}(z, f)=q_{0}(|z|).$ By Theorem~\ref{th1},
$f$ satisfies the inequality~(\ref{eq2*!A}) with $Q(z)=q_{0}(|z|)$
at any point $z_0\in \overline{\Bbb D}.$
 {\footnotesize

\medskip
\medskip
{\bf \noindent Ruslan Salimov} \\
Institute of Mathematics of NAS of Ukraine, \\
3 Tereschenkivska Str., 01 024 Kiev-4, UKRAINE\\
ruslan.salimov1@gmail.com

\medskip
{\bf \noindent Evgeny Sevost'yanov} \\
{\bf 1.} Zhytomyr Ivan Franko State University,  \\
40 Bol'shaya Berdichevskaya Str., 10 008  Zhytomyr, UKRAINE \\
{\bf 2.} Institute of Applied Mathematics and Mechanics\\
of NAS of Ukraine, \\
1 Dobrovol'skogo Str., 84 100 Slavyansk,  UKRAINE\\
esevostyanov2009@gmail.com

\medskip
{\bf \noindent Valery Targonskii} \\
Zhytomyr Ivan Franko State University,  \\
40 Bol'shaya Berdichevskaya Str., 10 008  Zhytomyr, UKRAINE \\
w.targonsk@gmail.com

}

\end{document}